\documentclass{article}
\usepackage[utf8]{inputenc}
\usepackage[english]{babel}
\usepackage{amscd,amssymb,amsxtra,amsmath,amsthm,amsfonts}
\usepackage{verbatim}
\usepackage{graphicx}
\usepackage{color}
\usepackage{tikz}
\usepackage{caption}
\usepackage[colorinlistoftodos]{todonotes}
\usetikzlibrary{arrows.meta}
\tikzstyle{point}=[ball color=white, circle, draw=black, inner sep=0.1cm]
\tikzstyle{pointblue}=[ball color=blue, circle, draw=black, inner sep=0.1cm]
\tikzstyle{pointred}=[ball color=red, circle, draw=black, inner sep=0.1cm]
\usepackage{authblk}

\newtheorem{theorem}{Theorem}[section]
\newtheorem{definition}[theorem]{Definition}

\newtheorem{proposition}[theorem]{Proposition}

\newtheorem{lemma}[theorem]{Lemma}
\newtheorem{Remark}[theorem]{Remark}

\DeclareMathOperator{\rad}{rad}
\DeclareMathOperator{\ecc}{ecc}

\title{Some considerations on the maximal safety distance in a graph}

\author[1]{Goran Erceg}
\author[1]{Aljoša Šubašić}
\author[,1]{Tanja Vojkovi\'c \thanks{Corresponding author: tanja@pmfst.hr, Rudjera Bo\v skovi\' ca 33, 21000 Split, Croatia}}
\affil[1]{Faculty of Science, University of Split}

\begin{document}
	
	\maketitle
	\renewcommand{\thefootnote}{\fnsymbol{footnote}}

\begin{abstract}
The work in this paper is motivated by I. Banič's and A. Taranenko's recent paper, where they introduced a new notion, the span of a graph. Their goal was to solve the problem of keeping a safety distance while two actors are moving through a graph and they present three different types of graph spans, depending on the movement rules. We observe the same goal, but give a different approach to that problem by directly defining the maximal safety distance for different movement rules two actors can take. This allowed us to solve several problems, prove some relations between different graph spans and calculate the span values for some classes of graphs.
\end{abstract}
	

	\section{Introduction}\label{intro}

	The inspiration for this work is found in Banič and Taranenko's recent paper, \cite{banic}, in which they present a graph theoretical equivalent to the Lelek's topological notion of a span, first defined in \cite{lelek}. It was motivated by the problem of keeping a safety distance between two actors moving in public spaces, the relevance of which has become quite apparent during the recent pandemic years. Many maps can be modeled as a graph, be it city streets, classrooms in school or rooms in the museum. The authors presented three different concepts of a span of a graph, depending on the movement rules the actors are bound by, and for each span they defined the vertex and the edge variant. They analyzed some families of graphs with regard to different spans and gave characterizations for some span values. In our considerations we observe the original problem of keeping a safety distance and we define this notion of a safety distance a bit differently then it was defined in \cite{banic}. \\
	Let us introduce the problem in more detail. The authors of the original paper used Alice and Bob as well known designations for two actors moving through a graph, so we will continue that approach. Let Alice and Bob be two actors moving through graph vertices via graph edges.\\
	Most importantly, Alice and Bob want to keep the maximal possible safety distance from each other in all steps.
	For example, one possible movement through graph $G$ in Figure \ref{pr1} where Alice and Bob visit all the vertices while keeping the distance $2$ from each other is presented in Table \ref{AB}.
	
	\begin{center}
	\captionof{table}{Alice's and Bob's order of movement}\label{AB} 
\begin{tabular}{ |c|c|c|c|c|c|c| } 
 \hline
  & 1 & 2 & 3 & 4 & 5 & 6 \\ 
 \hline
 Alice & $u_{4}$ & $u_{5}$ & $u_{6}$ & $u_{1}$ & $u_{2}$ & $u_{3}$ \\ 
 Bob & $u_{1}$ & $u_{2}$ & $u_{4}$ & $u_{3}$ & $u_{6}$ & $u_{5}$ \\ 
 \hline
\end{tabular}
\end{center}

\begin{center}
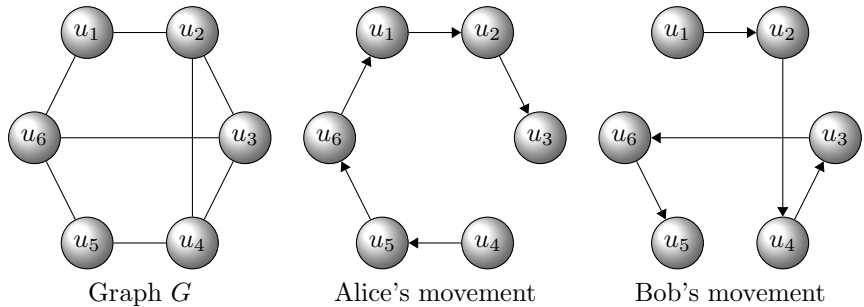

\begin{tabular}{ccc}
\begin{tikzpicture}[scale=0.7]
\node (1) at (1,4) [point] {$u_1$};
\node (2) at (3,4) [point] {$u_2$};
\node (3) at (4,2) [point] {$u_3$};
\node (4) at (3,0) [point] {$u_4$};
\node (5) at (1,0) [point] {$u_5$};
\node (6) at (0,2) [point] {$u_6$};
\draw (1) -- (2) -- (3) -- (4) -- (5) -- (6) -- (1)  (2) -- (4)  (3) -- (6);
\end{tikzpicture} &
\begin{tikzpicture}[scale=0.7]
\node (1) at (1,4) [point] {$u_1$};
\node (2) at (3,4) [point] {$u_2$};
\node (3) at (4,2) [point] {$u_3$};
\node (4) at (3,0) [point] {$u_4$};
\node (5) at (1,0) [point] {$u_5$};
\node (6) at (0,2) [point] {$u_6$};
\draw [-Triangle] (4)--(5);
\draw [-Triangle] (5)--(6);
\draw [-Triangle] (6)--(1);
\draw [-Triangle] (1)--(2);
\draw [-Triangle] (2)--(3);
\end{tikzpicture} & 
\begin{tikzpicture}[scale=0.7]
\node (1) at (1,4) [point] {$u_1$};
\node (2) at (3,4) [point] {$u_2$};
\node (3) at (4,2) [point] {$u_3$};
\node (4) at (3,0) [point] {$u_4$};
\node (5) at (1,0) [point] {$u_5$};
\node (6) at (0,2) [point] {$u_6$};
\draw [-Triangle] (1)--(2);
\draw [-Triangle] (2)--(4);
\draw [-Triangle] (4)--(3);
\draw [-Triangle] (3)--(6);
\draw [-Triangle] (6)--(5);
\end{tikzpicture}\\
Graph $G$ & Alice's movement & Bob's movement
\end{tabular}
\captionof{figure}{ Example of a graph with Alice's and Bob's movement keeping the distance $2$} \label{pr1}
\end{center}

	We will describe their movements as functions, from some set $N_{l}$, representing the order of their movement, to the set of vertices in a given graph. Of course, they can only move from vertex to vertex if they are adjacent. Generally, in each step, we will allow both actors to move to an adjacent vertex, or to stay in place, and our definitions of movement functions will reflect this.\\
	However, there are three possible ways in which Alice and Bob may move through a graph simultaneously:

	\begin{itemize}
	    \item Traditional movement rules: Alice and Bob move independently and at each step each of them may move to the adjacent vertex or stay in place.
	    \item Active movement rules: Both Alice and Bob must move to the adjacent vertex in each step.
	    \item Lazy movement rules: In each step, exactly one of the actors moves to an adjacent vertex, and the other actor stays in place.
	\end{itemize}
	These different rules, together with the requirement to visit all the vertices or all the edges, led the authors of \cite{banic} to define six different spans of a graph. Mathematical definitions of these spans are quite complicated and we present our approach to the problem in Section \ref{spans}. Intuitively, one can imagine the value of a particular span as the maximal possible safety distance that two actors can keep in a given graph with given movement rules. In this paper, we will observe vertex spans only, and we will use the same notation as used in \cite{banic}.\\
	For strong vertex span, corresponding to the traditional movement rules we use $\sigma^{\boxtimes}_{V}(G)$, for direct vertex span, corresponding to the active movement rules we use $\sigma^{\times}_{V}(G)$ and for Cartesian vertex span, corresponding to the lazy movement rules we use $\sigma^{\square}_{V}(G)$.
	
     Our main result, the relation between direct and Cartesian vertex span for each connected simple graph, is presented in Section \ref{odnosi}. Further, we observe values of spans for some classes of graphs, based on considerations from \cite{banic} and our own results, we prove and summarize this in Section \ref{classes}.

	\section{Preliminaries}\label{prelim}
	Let us give an overview of definitions and notation.
	Well known terminology of graph theory can be found in \cite{gross}, however, to facilitate reading, we will define some concepts important for our considerations.\\
	Let $G$ be a connected graph. The\textbf{ distance} $d_{G}(u,v)$, between the vertices $u,v\in V(G)$ is defined as the length of the shortest path between $u$ and $v$. 
	Let $u$ be a vertex in a simple connected graph $G$. The \textbf{eccentricity} of a vertex $u$ in $G$ is defined as
	$$\ecc_{G}(u)=\max\{d_{G}(u,v):v\in V(G)\}.$$
	If it is clear which graph we are observing, we use $\ecc(u)$ and $d(u,v)$.
	The \textbf{radius} of a simple connected graph $G$ is
	$$\rad(G)=\min\{\ecc(u):u\in V(G)\},$$
	and the \textbf{diameter} of a simple connected graph $G$ is
	$$\operatorname{diam}(G)=\max\{\ecc(u):u\in V(G)\}.$$
	Based on the motivation it is clear that it makes sense to observe only simple connected graphs, so the term graph will refer to a simple connected graph further on.\\
	Let $G$ and $H$ be graphs. With $H\subseteq G$ we denote that $H$ is a subgraph of $G$, and with $H\subseteq_{C} G$ we denote that $H$ is a connected subgraph of $G$.\\
	In paper \cite{banic} the authors make a connection between the maximal safety distance in a graph, as described in the Introduction, and particular subgraphs of graph products. They observe strong, Cartesian and direct graph products, corresponding to their definitions of strong, Cartesian and direct span. We will not repeat all their definitions here, but refer the reader to their paper as well as a book about graph products, for instance \cite{hammack}.
	Instead, we proceed to our own definitions of the maximal safety distance and later show how it is connected to the spans in paper \cite{banic}.
	
	With $\mathbb N_l$ we denote the set of first $l$ natural numbers, i.e. the set $\{1, \ldots, l \}.$

	\begin{definition}
	Let $G=(V,E)$ be a graph and $l\in\mathbb{N}$. We say that a surjective function $f_{l}:\mathbb{N}_{l}\longrightarrow V(G)$ is an\textbf{ $l$-track} on $G$ if $f(i)f(i+1)\in E(G)$ holds, for each $i\in\{1,...,l-1\}$.
	\end{definition}

    \begin{definition}
    Let $G=(V,E)$ be a graph and $l\in\mathbb{N}$. We say that a surjective function $f_{l}:\mathbb{N}_{l}\longrightarrow V(G)$ is a\textbf{ lazy $l$-track} on $G$ if $f(i)f(i+1)\in E(G)$ or $f(i)=f(i+1)$ holds, for each $i\in\{1,...,l-1\}$.
    \end{definition}

    It is easy to see that each $l$-track is also a lazy $l$-track. These functions represent a walk through all the vertices one actor takes around the graph, as described in the Introduction.\\
    One can immediately observe that if $f_{l}$ is a lazy $l$-track on some graph $G$, then $l\geq |V(G)|$ must hold.
    
    \begin{proposition}
    Let $G$ be a graph. There exists $l\in\mathbb{N}$ such that there exists a function $f_{l}:\mathbb{N}_{l}\longrightarrow V(G)$ which is a lazy $l$-track on $G$, and for each $l'>l$ there exists a lazy $l'$-track on $G$.
    \end{proposition}
    
    \begin{proof}
    The claim that there exists $l\in\mathbb{N}$ such that there exists a function $f_{l}:\mathbb{N}_{l}\longrightarrow V(G)$ that is a lazy $l$-track on $G$, follows directly from the fact that $G$ is connected so there exists a walk through all the vertices. The other part is easily seen.
    \end{proof}
    
    Now let us define the distance between two lazy $l$-tracks in a graph.
    
    \begin{definition}
    Let $G$ be a graph and $f,g$ lazy $l$-tracks on $G$. The\textbf{ distance between $f$ and $g$} is defined as
    $$m_{G}(f,g)=\min\{d_{G}(f(i),g(i)):i\in\mathbb{N}_{l}\}.$$
    \end{definition}
    
    \begin{lemma}\label{rad}
    Let $G$ be a graph and $f,g$ lazy $l$-tracks on $G$, for some $l$ for which such tracks exist. It holds
    $$m_{G}(f,g)\leq \rad(G).$$
    \end{lemma}
    
    \begin{proof}
    Let $G$ be a graph and $f,g$ lazy $l$-tracks on $G$. Let $v\in V(G)$ be a vertex for which $\ecc(v)=\rad(G)$ holds. From surjectivity of $f$ it follows that $i\in\mathbb{N}_{l}$ exists such that $f(i)=v$. But then we have $d_{G}(f(i),g(i))\leq \rad(G)$, so $m_{G}(f,g)\leq \rad (G)$ follows.
    \end{proof}
    
    \begin{Remark}
    Lemma \ref{rad} is stated and proved in paper \cite{banic} in different terms. Since we will use it on occasion, we considered it best to restate it here in the terms of lazy tracks.
    \end{Remark}

    \section{Maximal safety distances and graph spans}\label{spans}

    In this section we aim to define maximal safety distance formally and to show how these definitions correspond to the concept of graph spans.\\

   \textbf{Strong vertex span}\\
    
    Let $G$ be a graph and $l\in\mathbb{N}$ such that at least one lazy $l$-track exists on $G$. We denote
    
    $$Ms_{l}=\max\{m_{G}(f,g):f \text{ and } g \text{ } \text{ lazy }l  \text{-tracks on } G\}. $$
    
    The existence of $Ms_{l}$  follows from Lemma \ref{rad}.\\
    
    Let $G$ be a graph and let $S\subseteq\mathbb{N}$ be the set of all integers $l$ for which at least one lazy $l$-track exists on $G$. Obviously, $S$ is non-empty. We define
    
    $$Ms(G)=\max\{Ms_{l}:l\in S\}.$$

    This number is the maximal safety distance that can be kept while two actors walk in a graph with traditional movement rules.
    
    To show how $Ms(G)$ equals $\sigma^{\boxtimes}_{V}(G)$ we need one more definition, the same as used in paper \cite{banic}.
    \begin{definition}
    Let $G$ and $H$ be graphs such that $V(H)\subseteq V(G)\times V(G)$. We define
    $$\varepsilon_{G}(H)=\min\{d_{G}(u,v):(u,v)\in V(H)\}.$$
    \end{definition}
    
    \begin{theorem} \label{Mequalsigma}
    Let $G$ be a graph.
    It holds $Ms(G)=\sigma^{\boxtimes}_{V}(G)$.
    \end{theorem}
    \begin{proof}
    Let $G$ be a graph and let $S\subseteq\mathbb{N}$ be a set of all integers $l$ for which at least one lazy $l$-track exists on $G$.
    Let $$A=\{Ms_{l}:l\in S\}$$ and $$B=\{ \varepsilon_G (H) : H \subseteq _C G \boxtimes G \text{ with } p_1(V(H))=p_2(V(H))=V(G)\},$$
    where $\boxtimes$ is strong product of two graphs, \cite{banic}.
    By Theorem 3.3. in \cite{banic} we need to prove that $\max A = \max B.$
    We will prove that $A=B$. 
    Let $r\in A$. Let $f$ and $g$ be lazy $l$-tracks such that $m_G(f,g)=r.$ Denote $$V(H)=\{ (f(i), g(i)) : i \in \mathbb N_l \}$$ and $$ E(H)=\{ f(i)g(i) : i \in \mathbb N_l\}.$$
    From the definition of $f$ and $g$ it follows that $H$ is a connected subgraph of $G \boxtimes G$ and that $p_1(V(H))=p_2(V(H))=V(G)$. Therefore, $\varepsilon _G(H)=r$ so $r \in B.$
    Now, let $r\in B.$ Let $H \subseteq _C G \boxtimes G$ with  $p_1(V(H))=p_2(V(H))=V(G)$ such that $\varepsilon _G(H)=r$. \\
    
    We will now construct a walk through all the vertices in $H$ which will allow us to define lazy tracks in $G$ and conclude that $r\in A$. 
    Choose any $(u_1,v_1) \in V(H)$. We know that $H$ is connected and therefore we can obtain its spanning tree with root $(u_1,v_1)$. Now we construct a walk that visits all the vertices of $H$, starting from $(u_1,v_1)$ and we denote the steps of this walk with natural numbers $1,...,l$, so that we are in the vertex $(u_i,v_i)$ in step $i$. Note that, for some $i\neq j$, $(u_i,v_i)$ and $(u_j,v_j)$ may be the same.
    Also note that $l$ is not necessarily the smallest possible number of steps needed to visit all the nodes of $H$, but it holds $V(H)=\{ (u_i,v_i) : i \in \mathbb N_l \}$. Now we define functions $f : \mathbb N_l: \to V(G)$ and $g : \mathbb N_l: \to V(G)$ with $f(i)=u_i$ and $g(i)=v_i.$ From the definition of the strong product $\boxtimes$ and from $p_1(V(H))=p_2(V(H))=V(G)$ it follows that $f$ and $g$ are lazy $l$-tracks. Furthermore, we have
    \begin{align*}
        m_G(f,g) & = \min \{  d_G(f(i), g(i)) : i\in\mathbb{N}_{l}\} \\
        & = \min \{  d_G(u_i,v_i) : (u_i,v_i) \in V(H)\} \\
        & = \varepsilon_G(H).
    \end{align*}
    Therefore, $m_G(f,g)=r$ so $r \in A$. \\
    We have proved that $A=B$ so it follows that $\max A = \max B.$
    \end{proof}


       \textbf{Direct vertex span}\\
       
       Let $G$ be a graph and $l\in\mathbb{N}$ such that at least one $l$-track exists on $G$. We denote
    
    $$Md_{l}=\max\{m_{G}(f,g):f \text{ and } g \text{ } l  \text{-tracks on } G\}. $$
    
    The existence of this number follows from Lemma \ref{rad}.\\
    
    Let $G$ be a graph and let $D\subseteq\mathbb{N}$ be a set of all integers $l$ for which at least one $l$-track exists on $G$. We define
    
    $$Md(G)=\max\{Md_{l}:l\in D\}.$$
    
    This number is the maximal safety distance that can be kept while two actors walk in a graph with active movement rules. \\
    
    \begin{theorem}\label{md} Let $G$ be a graph.
    It holds $Md(G)=\sigma^{\times}_{V}(G)$.
    \end{theorem}
    \begin{proof}
    Proof is analogous to the proof of the Theorem \ref{Mequalsigma} so we omit it.
    \end{proof}
    \newpage
       \textbf{Cartesian vertex span}\\
       
       In order to formally describe lazy movement rules, we introduce the notion of opposite lazy $l$-tracks.\\
       
       \begin{definition}
       Let $G$ be a graph, $l\in\mathbb{N}$ and $f,g:\mathbb{N}_{l}\longrightarrow V(G)$ lazy $l$-tracks on $G$. We say that $f$ and $g$ are\textbf{ opposite lazy $l$-tracks} on $G$ if
       $$f(i)f(i+1)\in E(G) \text{ and } g(i)=g(i+1) \text{ or}$$
       $$g(i)g(i+1)\in E(G) \text{ and } f(i)=f(i+1).$$
       \end{definition}
       
       Let $G$ be a graph and $l\in\mathbb{N}$ such that at least one pair of opposite lazy $l$-tracks exists on $G$. For $l\in\mathbb{N}$, let $A_{l}$ be the set of all pairs of opposite lazy $l$-tracks on $G$. We denote
    $$Mc_{l}=\max\{m_{G}(f,g):(f,g)\in A_{l}\}.$$
    The existence of this number follows from Lemma \ref{rad}.\\
    
    Let $G$ be a graph and let $C\subseteq\mathbb{N}$ be a set of all integers $l$ for which $A_{l}$ is a non-empty set. We define
    $$Mc(G)=\max\{Mc_{l}:l\in C\}.$$

    This number is the maximal safety distance that can be kept while two actors walk in a graph with lazy movement rules. 
    
    \begin{theorem}\label{mc} Let $G$ be a graph.
    It holds $Mc(G)=\sigma^{\square}_{V}(G)$.
    \end{theorem}
    \begin{proof}
    Proof is analogous to the proof of the Theorem \ref{Mequalsigma} so we omit it.
    \end{proof}
    
We have shown that the definitions of maximal safety distance  presented here indeed define vertex spans presented in \cite{banic} so further on we only use the span notation, $\sigma^{\boxtimes}_{V}(G)$, $\sigma^{\times}_{V}(G)$ and $\sigma^{\square}_{V}(G)$.

\begin{Remark}
As we said in the Introduction, some claims will be proven by describing Alice's and Bob's movement in the graph, instead of formally defining $l$-tracks for some $l\in\mathbb{N}$. In those cases we will use the expressions such as Alice's and Bob's lazy walks, walks, or opposite lazy walks in a graph, which correspond to traditional, active and lazy movement rules, respectively. This approach is used to make the proofs more fluid and more understandable to the reader.
\end{Remark}


\section{Relation between different vertex spans}\label{odnosi}

Here we take a look at the relations between strong, direct and Cartesian vertex span.

\begin{proposition}\label{strong}
Let $G$ be a graph. It holds $\sigma^{\boxtimes}_V(G)\geq \max\{\sigma^{\times}_V(G),\sigma^{\square}_V(G)\}$.
\end{proposition}
\begin{proof}
The claim easily follows from the definitions of $Ms(G)$, $Md(G)$, $Mc(G)$ and Theorems \ref{Mequalsigma}, \ref{md} and \ref{mc}.
\end{proof}

\begin{theorem}\label{razlika}
Let $G$ be a graph. Then $$|\sigma^{\times}_V(G)-\sigma^{\square}_V(G)| \leq 1.$$
\end{theorem}
\begin{proof}
For any graph $G$, we will prove the following: 
\begin{itemize}
    \item [(1)]$ \sigma^{\square}_V(G) \geq \sigma^{\times}_V(G)-1$;
    \item [(2)]$ \sigma^{\times}_V(G) \geq \sigma^{\square}_V(G)-1$.
\end{itemize}
Proof of (1):\\
Let $\sigma^{\times}_V(G)=r$. Then there exist $l$-tracks $f$ and $g$ such that for each $i \in \mathbb N_l$, $d(f(i),g(i)) \geq r$. Let us define lazy $(2l-1)$-tracks $f'$ and $g'$ in the following manner:
$$f'(i)=f(\left\lceil\frac{i+1}{2}\right\rceil);$$
$$g'(i)=g(\left\lceil\frac{i}{2}\right\rceil).$$
We claim that $f'$ and $g'$ are opposite lazy tracks and that for each $i \in \mathbb N_{2l-1}$, $d(f'(i),g'(i)) \geq r-1$.
First, let us prove that $f'$ and $g'$ are opposite lazy tracks.\\
For each odd $i$, i.e., $i=2k-1$, $f'(i)f'(i+1)$ is in $E(G)$ because
$$f'(i)f'(i+1)=f(\left\lceil\frac{2k}{2}\right\rceil)f(\left\lceil\frac{2k+1}{2}\right\rceil=f(k)f(k+1) \in E(G);$$ and it holds
$$g'(i)=g(\left\lceil\frac{2k-1}{2}\right\rceil)=g(k)=g(\left\lceil\frac{2k}{2}\right\rceil)=g'(i+1).$$
Similarly, for each even $i$, i.e., $i=2k$, it holds
$$f'(i)=f(\left\lceil\frac{2k+1}{2}\right\rceil)=f(k+1)=f(\left\lceil\frac{2k+2}{2}\right\rceil)=f'(i+1);$$ and $g'(i)g'(i+1)$ is in $E(G)$:
$$g'(i)g'(i+1)=g(\left\lceil\frac{2k}{2}\right\rceil)g(\left\lceil\frac{2k+1}{2}\right\rceil)=g(k)g(k+1) \in E(G).$$
Now, let us prove that for each $i \in \mathbb N_{2l-1}$, $d(f'(i),g'(i)) \geq r-1$.\\
Again, for each odd $i$, i.e., $i=2k-1$, it holds:
$$d(f'(i),g'(i))=d(f(\left\lceil\frac{2k}{2}\right\rceil),g(\left\lceil\frac{2k-1}{2}\right\rceil))
=d(f(k),g(k)) \geq r.$$
Similarly, for each even $i$, i.e., $i=2k$, it holds:
$$d(f'(i),g'(i))=d(f(\left\lceil\frac{2k+1}{2}\right\rceil),g(\left\lceil\frac{2k}{2}\right\rceil))
=d(f(k+1),g(k)).$$
Since $f(k)f(k+1)\in E(G)$, then $$d(f(k+1),g(k)) \geq d(f(k),g(k))-1 \geq r-1.$$
Proof of (2):\\
Let $\sigma^{\square}_V(G)=r$. Then there exist opposite lazy $l$-tracks $f$ and $g$ such that for each $i \in N_l$, $d(f(i),g(i)) \geq r$.\\
Since, $f$ and $g$ are opposite, we will first define a $(l-1)$-sequence $X$ of $1$'s and $2$'s in a following manner:
$$
X(i):=
\left\{
	\begin{array}{ll}
		1  & \mbox{if } f(i)f(i+1) \in E(G)\\
		2  & \mbox{if } g(i)g(i+1) \in E(G)\\
	\end{array}
\right.\\
$$
Now, we separate the sequence $X$ into pairs $(X(2k-1),X(2k))$ for $k  \in \mathbb N_{\left\lfloor\frac{l-1}{2}\right\rfloor}$. If $l-1$ is odd we will leave the element $X(l-1)$ unpaired.
All the pairs must be elements of $\{(1,1),(1,2),(2,1),(2,2)\}$. Let $a$ be the number of $(1,2)$ and $(2,1)$ pairs combined. We now define $(l-a)$-tracks $f'$ and $g'$ using the following algorithm:

We first define $f'(1)=f(1)$ and $g'(1)=g(1)$, and set the value of $b$ to $0$ and the value of $i$ to $1$. 
Now, we go through pairs $(X(2k-1),X(2k))$ for $k  \in \mathbb N_{\left\lfloor\frac{l-1}{2}\right\rfloor}$ one by one and do the following:\\
\indent If the pair is $(1,1)$ we increase the value of $i$ by two and define
$$f'(i-1)=f(i-1+b); f'(i)=f(i+b),$$ and also
$$g'(i-1)=x; g'(i)=g(i+b),$$ where $x$ is any neighbour of $g(i+b)$.\\
\indent Similarly, if the pair is $(2,2)$ we increase the value of $i$ by two and define
$$f'(i-1)=x; f'(i)=f(i+b),$$ and also
$$g'(i-1)=g(i-1+b); g'(i)=g(i+b),$$ where $x$ is any neighbour of $f(i+b)$.\\
\indent If the pair is $(1,2)$ or $(2,1)$ we increase the value of $b$ and $i$ by one and define
$$f'(i)=f(i+b); g'(i)=g(i+b).$$
\indent Lastly, if we have any unpaired $X(l-1)$, we increase the value of $i$ by one and define
$$f'(i)=f(i+b); g'(i)=x,$$
if $X(l-1)=1$, or
$$f'(i)=y; g'(i)=g(i+b),$$
if $X(l-1)=2$, where $x$ and $y$ are any neighbours of $g(i+b)$ and $f(i+b)$ respectfully.
By construction, we see that $f'$ and $g'$ are $(l-a)$-tracks and that for each $i \in \mathbb N_{l-a}$, $d(f'(i),g'(i)) \geq r-1$. \\
Now from (1) and (2) we have $$\sigma^{\times}_V(G)-1 \leq \sigma^{\square}_V(G) \leq \sigma^{\times}_V(G)+1,$$
which ends our proof.
\end{proof}

\begin{Remark}
To fully clarify the second, more algorithmic, part of the proof of Theorem \ref{razlika}, we will additionally explain it in terms of Alice and Bob. When Alice and Bob are walking in opposite lazy walks, while one of them is walking, the other is standing still and vice versa. That means, if we have their lazy walks that keep the safety distance $r$, we can first derive the order in which the two of them walked through all the vertices. When constructing their walks in active movement rules we are starting from the same position they started in lazy movement rules and then we go through the order in which they moved. Two different situations may happen. Either one of them moved, and then the other one, or one of them moved twice in a row while the other stood still for two steps. In case one of them moved and then the other one, those two movements will just happen simultaneously in active movement rules, and Alice and Bob will end up in a position they were in with lazy movement rules, still at safety distance $r$. In case of one of them moving twice in a row, while the other stands still, we can transfer those movements to active movement rules just by moving the one who stood still to any neighbour vertex and back, thus possibly lowering the safety distance by $1$, but still keeping it at least $r-1$. 
\end{Remark}

Let us also notice that the bound in Theorem \ref{razlika} is tight. For example, let us observe the two graphs in Figure \ref{pr2}. 

\begin{center}
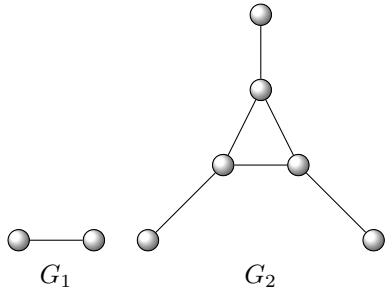

\begin{tabular}{cc}
\begin{tikzpicture}
\node (1) at (0,0) [point] {};
\node (2) at (1,0) [point] {};
\draw (1) -- (2);
\end{tikzpicture}
&
\begin{tikzpicture}
\node (1) at (0,0) [point] {};
\node (2) at (1,1) [point] {};
\node (3) at (2,1) [point] {};
\node (4) at (3,0) [point] {};
\node (5) at (1.5,2) [point] {};
\node (6) at (1.5,3) [point] {};
\draw (1) -- (2) -- (3) -- (4) (2) -- (5) (3) -- (5) (5) -- (6);
\end{tikzpicture}\\
$G_{1}$ & $G_{2}$
\end{tabular} 

\captionof{figure}{ Graphs for which the bounds of Theorem \ref{razlika} are obtained} 
\label{pr2}
\end{center}
It can be easily seen that $\sigma^{\times}_V(G_{1})=1$ and $\sigma^{\square}_V(G_{1})=0$. Similarly, we can easily see that $\sigma^{\times}_V(G_{2})=1$ and $\sigma^{\square}_V(G_{2})=2$.


\section{Spans of some graph classes}\label{classes}
	
First we take a look at the way a cut-edge in a graph might help us reduce the problem of calculating the span values of a graph to its subgraphs determined by cut-edge. This result might lead to an algorithm for calculating upper bond for spans, for some graphs better then the radius. We used the reasoning from proof of Proposition \ref{cut} to calculate some results for this paper, for instance in Proposition \ref{najmanji}.

\begin{proposition}\label{cut}
Let $G$ be a graph such that $|V(G)| \geq 3$. Let $xy$ be a cut-edge in $G$ and $G_1$ and $G_2$ subgraphs of $G$ induced by components of $G-xy$ such that $x \in V(G_1)$ and $y \in V(G_2)$. Then $$\sigma^{\boxtimes}_V(G) \leq \max\{\ecc_{G_1}(x),\ecc_{G_2}(y)\}.$$
\end{proposition}
\begin{proof}
We will provide this proof in terms of Alice and Bob. Depending on where they start their walks we have three different situations.\\
Situation 1: Both Alice and Bob start their walks in some vertex of $G_1$.
In that case, at some point, both of them will have to visit vertices of $G_2$, and the only way to go there is through the cut-edge $xy$. Let us say, without any loss of generality,  that Alice leaves $G_1$ first. When she gets to vertex $x$, Bob is still in $G_1$, so at that point their distance is for sure less than or equal to $\ecc_{G_1}(x)$. So our claim holds.\\
Situation 2: Both Alice and Bob start their walks in some vertex of $G_2$.
Analogously as in situation 1, we see that, at some point, when leaving $G_2$, the distance between Alice and Bob will have to be less than or equal to $\ecc_{G_2}(y)$.\\
Situation 3: Alice and Bob start their walks in vertices of different subgraphs $G_1$ and $G_2$. Let us say, without any loss of generality, that Alice starts in $G_1$, and Bob in $G_2$. Either one of them will walk through cut-edge $xy$ to the other subgraph, at which point we will end up in situation 1 or 2, or they will both switch to the other subgraph at the same time, in which case, they will have to be in vertices $x$ and $y$ at the same step, and at that point their distance will be $1$, which is less then or equal to $\max\{\ecc_{G_1}(x),\ecc_{G_2}(y)\}$, since, because of $|V(G)| \geq 3$, at least one of the graphs $G_1$ and $G_2$ is non-trivial.
\end{proof}

\begin{Remark}
Since we already had that $\sigma^{\boxtimes}_V(G) \leq \rad(G)$, one might wonder if this is an improvement of that result and how does $\max\{\ecc_{G_1}(x),\ecc_{G_2}(y)\}$ compare to $rad(G)$. The two measures are incomparable as there are examples of graphs where either one is greater than the other, as is shown in Figure \ref{remark}. So this is an additional result that one might use when calculating upper bonds for spans of certain types of graphs.\\
\begin{center}
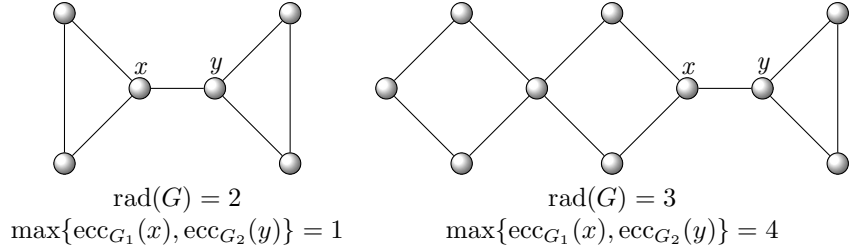

    \begin{tabular}{cc}
\begin{tikzpicture}
\node (1) at (0,0) [point] {};
\node (2) at (0,2) [point] {};
\node (3) at (3,0) [point] {};
\node (4) at (3,2) [point] {};
\node (5) at (1,1) [point] {};
\node (6) at (2,1) [point] {};
\node[above, outer sep=2pt] at (1,1) {x};
\node[above, outer sep=2pt] at (2,1) {y};
\draw (5) -- (2) -- (1) -- (5) -- (6) -- (4) -- (3) -- (6);
\end{tikzpicture}
&
\begin{tikzpicture}
\node (1) at (1,0) [point] {};
\node (2) at (3,0) [point] {};
\node (3) at (6,0) [point] {};
\node (4) at (0,1) [point] {};
\node (5) at (2,1) [point] {};
\node (6) at (4,1) [point] {};
\node (7) at (5,1) [point] {};
\node (8) at (1,2) [point] {};
\node (9) at (3,2) [point] {};
\node (10) at (6,2) [point] {};
\node[above, outer sep=2pt] at (4,1) {x};
\node[above, outer sep=2pt] at (5,1) {y};
\draw (6) -- (2) -- (5) -- (1) -- (4) -- (8) -- (5) -- (9) -- (6) -- (7) -- (10) -- (3) -- (7);
\end{tikzpicture}
\\
$\rad(G)=2$
& 
$\rad(G)=3$
\\
$\max\{\ecc_{G_1}(x),\ecc_{G_2}(y)\}=1$
&
$\max\{\ecc_{G_1}(x),\ecc_{G_2}(y)\}=4$
\\
    \end{tabular}
    \captionof{figure}{Graphs that illustrate the difference between the radius of a graph and the maximum of the eccentricities of the ends of a cut-edge}\label{remark} 
\end{center}
\end{Remark}


In further observations of spans of different graph families we mostly encountered graphs for which direct vertex span was greater than or equal to their Cartesian vertex span. So it prompted us to to find 
a family of graphs for which a greater safety distance can be obtained following lazy movement rules rather then active ones. We found such family and to describe it, we first define a paramecium graph. 

\begin{definition}\label{param}
Let $C_n$ be a cycle with $n$ vertices $\{v_1,...,v_n\}$, so that $v_iv_{i+1}\in E(C_{n})$, $i\in\{1,...,n-1\}$ and $v_nv_1\in E(C_{n})$. A graph obtained by adding $n$ leaves $\{u_1,...,u_n\}$ to $C_n$ in a way that $u_i$ and $v_i$ are adjacent for each $i\in\mathbb{N}_n$ will be called \textbf{a paramecium graph}. We will denote it by $PC_n$. 
\end{definition}
It is easily seen that $|V(PC_n)|=2n$ and $\rad(PC_n)=\rad(C_n)+1=\lfloor\frac{n}{2}\rfloor +1$.

\begin{theorem}\label{paramecium}
It holds
    $$\sigma^{\times}_V(PC_{n})=\left \lfloor\frac{n}{2} \right \rfloor$$ and $$\sigma^{\square}_V(PC_{n})=\sigma^{\boxtimes}_V(PC_{n})=\left  \lceil\frac{n}{2} \right \rceil. $$
\end{theorem}
\begin{proof}
Let $PC_n$ be a paramecium graph with vertex labels as in Definition \ref{param}. We will refer to vertices $\{u_1,...,u_n\}$ as the pendant vertices. As we can see from the claim, the value of different spans will depend on the parity of number $n$, so our proof will have two cases, that is, we will prove:\\

For $k \geq 1$ it holds $\sigma^{\square}_V(PC_{2k+1})=\sigma^{\boxtimes}_V(PC_{2k+1})=k+1$ and $\sigma^{\times}_V(PC_{2k+1})=k$;\\

For $k \geq 2$ it holds
$\sigma^{\times}_V(PC_{2k})=\sigma^{\square}_V(PC_{2k})=\sigma^{\boxtimes}_V(PC_{2k})=k$.\\

Let us prove the claim for $PC_{2k+1}$. We have $\sigma^{\square}_V(PC_{2k+1}) \leq \rad(PC_{2k+1})=\lfloor\frac{2k+1}{2}\rfloor +1=k+1$.
Now, let us show that $\sigma^{\square}_V(PC_{2k+1})=k+1$. We will prove this by describing Alice's and Bob's movements in $PC_{2k+1}$ that follow lazy movement rules, which keep the distance $k+1$.
First, observe that for each pendant vertex $u_i$, there are two cycle vertices that are on the distance $k+1$ from it, they are $v_m$ and $v_{m+1}$, where $m=(i+k)\mod (2k+1)$. Let us describe first few Alice's and Bob's steps with Table \ref{tbl1}.
\begin{center}
\begin{tabular}{ |c|c|c|c|c|c|c|c|} 
 \hline
  & 1 & 2 & 3 & 4 & 5 & 6 & 7 \\ 
 \hline
 Alice's position & $u_{1}$ & $u_{1}$ & $u_{1}$ & $u_{1}$ & $v_{1}$ & $v_{2}$ & $u_{2}$\\ 
Bob's position & $u_{k+1}$ & $v_{k+1}$ & $v_{k+2}$ & $u_{k+2}$ & $u_{k+2}$ & $u_{k+2}$ & $u_{k+2}$\\
 \hline
\end{tabular}

\captionof{table}{ First few Alice's and Bob's steps in $PC_{2k+1}$ keeping the distance $k+1$ with lazy movement rules} 
\label{tbl1}
\end{center}
In Figure \ref{neparni} we can see the illustration of these steps for $PC_5$, for which $k=2$.
It is easy to see that, by continuing this way, Alice and Bob will visit all the vertices, while always keeping the distance greater or equal to $k+1$.\\
\begin{center}
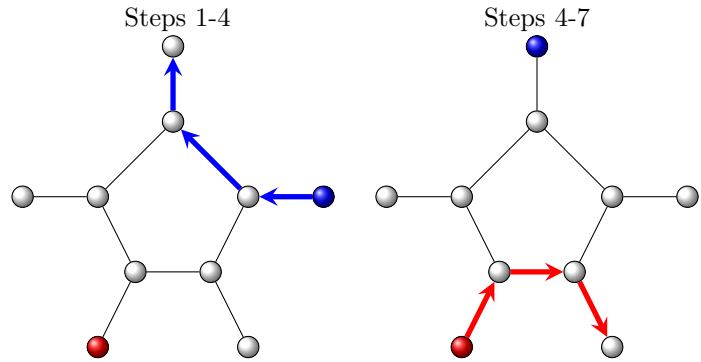

\begin{tabular}{cc}
Steps $1$-$4$ & Steps $4$-$7$\\
\begin{tikzpicture}
\node (1) at (1,0) [pointred] {};
\node (2) at (3,0) [point] {};
\node (3) at (1.5,1) [point] {};
\node (4) at (2.5,1) [point] {};
\node (5) at (0,2) [point] {};
\node (6) at (1,2) [point] {};
\node (7) at (3,2) [point] {};
\node (8) at (4,2) [pointblue] {};
\node (9) at (2,3) [point] {};
\node (10) at (2,4) [point] {};
\draw[-stealth, blue, line width=2pt] (8) -- (7);
\draw[-stealth, blue, line width=2pt] (7) -- (9);
\draw[-stealth, blue, line width=2pt] (9) -- (10);
\draw (1) -- (3) -- (4) (2) -- (4) -- (7) (9) -- (6) -- (5) (6) -- (3);
\end{tikzpicture} \qquad
&
\begin{tikzpicture}
\node (1) at (1,0) [pointred] {};
\node (2) at (3,0) [point] {};
\node (3) at (1.5,1) [point] {};
\node (4) at (2.5,1) [point] {};
\node (5) at (0,2) [point] {};
\node (6) at (1,2) [point] {};
\node (7) at (3,2) [point] {};
\node (8) at (4,2) [point] {};
\node (9) at (2,3) [point] {};
\node (10) at (2,4) [pointblue] {};
\draw[-stealth, red, line width=2pt] (1) -- (3);
\draw[-stealth, red, line width=2pt] (3) -- (4);
\draw[-stealth, red, line width=2pt] (4) -- (2);
\draw (8) -- (7) -- (9) (10) -- (9) -- (6) (5) -- (6) -- (3) (4) -- (7);
\end{tikzpicture}\\
\end{tabular}
\captionof{figure}{ First $7$ steps of Alice and Bob in $PC_5$, following lazy movement rules and keeping the distance at least $3$} \label{neparni}
\end{center}

We have shown that $\sigma^{\square}_V(PC_{2k+1})=k+1$ and from Proposition \ref{strong} it follows $\sigma^{\boxtimes}_V(PC_{2k+1})=k+1$.\\
Let us now prove that $\sigma^{\times}_V(PC_{2k+1})=k$. First, we will show that $\sigma^{\times}_V(PC_{2k+1})<k+1$. We know that $\rad(C_{2k+1})=k$, so if at any point Alice and Bob are both in cycle vertices, their distance is less then $k+1$. Let us assume they can keep the distance $k+1$ at all times. If Alice and Bob both start in pendant vertices, in the next step they are both in cycle vertices, so the only possible start is for one of them to be in a cycle vertex, and for the other to be in a pendant vertex. Moreover, since we are trying to keep the maximal distance, they have to start at the diametrically opposite sides of the cycle. Formally, and without the loss of generality, let Alice start in the pendant vertex $u_1$ and let Bob start in the cycle vertex $v_{k+1}$, since $d(u_1,v_{k+1})=k+1$. In the next step, Alice has only one option, to move to $v_1$, and Bob has $3$ options. If he moves to $v_{k}$ or $v_{k+2}$, they are both in the cycle vertices, and the third option takes him to $u_{k+1}$, which leaves them in an analogous situation as in the starting positions. So $\sigma^{\times}_V(PC_{2k+1})<k+1$. Now the fact that $\sigma^{\times}_V(PC_{2k+1})=k$ follows from $\sigma^{\square}_V(PC_{2k+1})=k+1$, $\sigma^{\times}_V(PC_{2k+1})<k+1$ and Theorem \ref{razlika}.\\
Let us prove the claim for $PC_{2k}$.  Its radius is $k+1$, so that is the maximal possible span value for each of the spans. Since $\rad(C_{2k})=k$ we immediately see that for all three movement rules, if at any point Alice and Bob are both in the vertices of the cycle, their distance will be at most $k$.\\ 
For active movement rules, the proof that $\sigma^{\times}_V(PC_{2k})<k+1$ is the same as for $PC_{2k+1}$. To show that $\sigma^{\times}_V(PC_{2k})=k$ let us describe Alice's and Bob's movements in $PC_{2k}$, following active movement rules and keeping the distance $k$. We will describe the first few steps with Table \ref{tbl2}.

\begin{center}
\begin{tabular}{ |c|c|c|c|c|c|c| } 
 \hline
 & 1 & 2 & 3 & 4 & 5 & 6 \\ 
 \hline
 Alice's position &  $u_{1}$ & $v_{1}$ & $v_{2}$ & $u_{2}$ & $v_{2}$ & $v_{3}$\\ 
Bob's position & $u_{k+1}$ & $v_{k+1}$ & $v_{k+2}$ & $u_{k+2}$ & $v_{k+2}$ & $v_{k+3}$\\ 
 \hline
\end{tabular}
\captionof{table}{ First few Alice's and Bob's steps in $PC_{2k}$ keeping the distance $k$ with active movement rules}
\label{tbl2}
\end{center}
By continuing in this way, Alice and Bob will visit all the vertices, keeping their distance, so we have proven $\sigma^{\times}_V(PC_{2k})=k$. The illustration of first $3$ steps in $PC_6$ is given in Figure \ref{parni}.
\begin{center}
\begin{tikzpicture}[scale=0.8]
\node (1) at (1,0) [pointred] {};
\node (2) at (3,0) [point] {};
\node (3) at (1.5,1) [point] {};
\node (4) at (2.5,1) [point] {};
\node (5) at (0,2) [point] {};
\node (6) at (1,2) [point] {};
\node (7) at (3,2) [point] {};
\node (8) at (4,2) [point] {};
\node (9) at (1.5,3) [point] {};

\node (10) at (2.5,3) [point] {};
\node (11) at (1,4) [point] {};
\node (12) at (3,4) [pointblue] {};
\draw (3) -- (4) -- (7) -- (10) -- (9) -- (6) -- (3) (2) -- (4) (5) -- (6) (7) -- (8) (9) -- (11);
\draw[-stealth, red, line width=2pt] (1) -- (3);
\draw[-stealth, blue, line width=2pt] (12) -- (10);
\end{tikzpicture}
\qquad
\begin{tikzpicture}[scale=0.8]
\node (1) at (1,0) [point] {};
\node (2) at (3,0) [point] {};
\node (3) at (1.5,1) [pointred] {};
\node (4) at (2.5,1) [point] {};
\node (5) at (0,2) [point] {};
\node (6) at (1,2) [point] {};
\node (7) at (3,2) [point] {};
\node (8) at (4,2) [point] {};
\node (9) at (1.5,3) [point] {};
\node (10) at (2.5,3) [pointblue] {};
\node (11) at (1,4) [point] {};
\node (12) at (3,4) [point] {};
\draw (4) -- (7) -- (10) (9) -- (6) -- (3) (1) -- (3) (2) -- (4) (5) -- (6) (7) -- (8) (9) -- (11) (10) -- (12);
\draw[-stealth, red, line width=2pt] (3) -- (4);
\draw[-stealth, blue, line width=2pt] (10) -- (9);
\end{tikzpicture}
\qquad
\begin{tikzpicture}[scale=0.8]
\node (1) at (1,0) [point] {};
\node (2) at (3,0) [point] {};
\node (3) at (1.5,1) [point] {};
\node (4) at (2.5,1) [pointred] {};
\node (5) at (0,2) [point] {};
\node (6) at (1,2) [point] {};
\node (7) at (3,2) [point] {};
\node (8) at (4,2) [point] {};
\node (9) at (1.5,3) [pointblue] {};
\node (10) at (2.5,3) [point] {};
\node (11) at (1,4) [point] {};
\node (12) at (3,4) [point] {};
\draw (3) -- (4) -- (7) -- (10) -- (9) -- (6) -- (3) (1) -- (3) (5) -- (6) (7) -- (8) (10) -- (12);
\draw[-stealth, red, line width=2pt] (4) -- (2);
\draw[-stealth, blue, line width=2pt] (9) -- (11);
\end{tikzpicture}

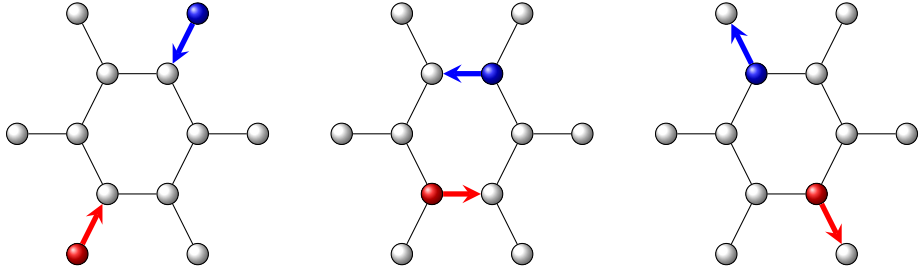
\captionof{figure}{ First $3$ steps of Alice and Bob in $PC_6$, following active movement rules and keeping the distance at least $3$} \label{parni}
\end{center}

Let us now observe traditional movement rules. The situation is a little bit different, since it doesn't necessarily have to occur that both of them are in the cycle vertices at the same time. Let us show that $\sigma^{\boxtimes}_V(PC_{2k})<k+1$. Let us assume the opposite, that there are lazy walks that Alice and Bob can take in order to keep the distance $k+1$. Without the loss of generality, let Alice start in the pendant vertex $u_1$. That leaves Bob starting in either the cycle vertex $v_{k+1}$ or the pendant  vertex $u_{k+1}$. It can be easily seen that neither of them can ever move to any cycle vertex other than $v_1$ and $v_{k+1}$ without getting closer than $k+1$ to one another. It immediately follows that $\sigma^{\square}_V(PC_{2k})<k+1$. To show that $\sigma^{\square}_V(PC_{2k})=k$, we can easily construct Alice's and Bob's walks, similar as with the active movement rules, just following lazy movement rules, so each of the actors doesn't move in every other step. So $\sigma^{\square}_V(PC_{2k})=k$, which also gives us $\sigma^{\boxtimes}_V(PC_{2k})=k$.
\end{proof}

Also, worth noticing is that graph $PC_3$ is the graph of the smallest order for which its Cartesian vertex span is greater than its direct vertex span. Let us prove that claim.

\begin{proposition}\label{najmanji}
Graph $PC_3$ is the graph of the smallest order such that its Cartesian vertex span is greater than its direct vertex span.
\end{proposition}
\begin{proof}
Let $G$ be a graph such that its Cartesian vertex span is greater than its direct vertex span. Since for any graph, for which its direct vertex span equals 0, it follows that its Cartesian vertex span is also 0, then $\sigma^{\times}_V(G)$ must be at least $1$ and $\sigma^{\square}_V(G)$ at least $2$. All connected graphs up to $3$ vertices have radii at most $1$. Only connected graphs with $4$ vertices of radius $2$ are $P_4$ ($\sigma^{\times}_V(P_4) = 1$, $\sigma^{\square}_V(P_4) = 0$) and $C_4$ ($\sigma^{\times}_V(C_4) = 2$, $\sigma^{\square}_V(C_4) = 1$).\\
Lastly, with $5$ vertices, we have $10$, up to isomorphism, connected graphs of radius $2$: $P_5$ ($\sigma^{\times}_V(P_5) = 1$, $\sigma^{\square}_V(P_5) = 0$), $C_5$ ($\sigma^{\times}_V(C_5) = 2$, $\sigma^{\square}_V(C_5) = 2$), and the remaining $8$ are presented in Figure \ref{pr3}, alongside their respective direct and Cartesian vertex spans:\\
\begin{center}
    \begin{tabular}{cccc}
\begin{tikzpicture}
\node (1) at (0,0) [point] {};
\node (2) at (1,1) [point] {};
\node (3) at (0,2) [point] {};
\node (4) at (2,2) [point] {};
\node (5) at (2,0) [point] {};
\draw (1) -- (2) -- (3) -- (1) (2) -- (4) -- (5);
\end{tikzpicture}
&
\begin{tikzpicture}
\node (1) at (0,0) [point] {};
\node (2) at (1,1) [point] {};
\node (3) at (0,2) [point] {};
\node (4) at (2,2) [point] {};
\node (5) at (2,0) [point] {};
\draw (1) -- (2) -- (3) -- (1) (3) -- (4) (1) -- (5);
\end{tikzpicture}
& 
\begin{tikzpicture}
\node (1) at (0,0) [point] {};
\node (2) at (1,1) [point] {};
\node (3) at (0,2) [point] {};
\node (4) at (2,2) [point] {};
\node (5) at (2,0) [point] {};
\draw (1) -- (2) -- (3) -- (1) (3) -- (4) (1) -- (5) (4) -- (5);
\end{tikzpicture}
&
\begin{tikzpicture}
\node (1) at (0,0) [point] {};
\node (2) at (1,1) [point] {};
\node (3) at (0,2) [point] {};
\node (4) at (2,2) [point] {};
\node (5) at (2,0) [point] {};
\draw (1) -- (2) (3) -- (1) (3) -- (4) (1) -- (5) (4) -- (5);
\end{tikzpicture}\\
$\sigma^{\times}_V(G) = 1$ &
$\sigma^{\times}_V(G) = 1$ &
$\sigma^{\times}_V(G) = 2$ &
$\sigma^{\times}_V(G) = 2$\\
$\sigma^{\square}_V(G) = 1$ &
$\sigma^{\square}_V(G) = 1$ &
$\sigma^{\square}_V(G) = 1$ &
$\sigma^{\square}_V(G) = 1$\\
\begin{tikzpicture}
\node (1) at (0,0) [point] {};
\node (2) at (1,1) [point] {};
\node (3) at (0,2) [point] {};
\node (4) at (2,2) [point] {};
\node (5) at (2,0) [point] {};
\draw (1) -- (2) -- (3) -- (1) (3) -- (4) (1) -- (5) (2) -- (5);
\end{tikzpicture}
&
\begin{tikzpicture}
\node (1) at (0,0) [point] {};
\node (2) at (1,1) [point] {};
\node (3) at (0,2) [point] {};
\node (4) at (2,2) [point] {};
\node (5) at (2,0) [point] {};
\draw (1) -- (2) (3) -- (1) (3) -- (4) (1) -- (5);
\end{tikzpicture}
&
\begin{tikzpicture}
\node (1) at (0,0) [point] {};
\node (2) at (1,1) [point] {};
\node (3) at (0,2) [point] {};
\node (4) at (2,2) [point] {};
\node (5) at (2,0) [point] {};
\draw (1) -- (5) -- (4) -- (3) -- (1) (1) -- (2) -- (4);
\end{tikzpicture}
&
\begin{tikzpicture}
\node (1) at (0,0) [point] {};
\node (2) at (1,1) [point] {};
\node (3) at (0,2) [point] {};
\node (4) at (2,2) [point] {};
\node (5) at (2,0) [point] {};
\draw (1) -- (2) -- (3) -- (1) (3) -- (4) (1) -- (5) (2) -- (5) (4) -- (5);
\end{tikzpicture}\\
$\sigma^{\times}_V(G) = 1$ &
$\sigma^{\times}_V(G) = 1$ &
$\sigma^{\times}_V(G) = 2$ &
$\sigma^{\times}_V(G) = 2$\\
$\sigma^{\square}_V(G) = 1$ &
$\sigma^{\square}_V(G) = 1$ &
$\sigma^{\square}_V(G) = 1$ &
$\sigma^{\square}_V(G) = 1$
    \end{tabular}
    
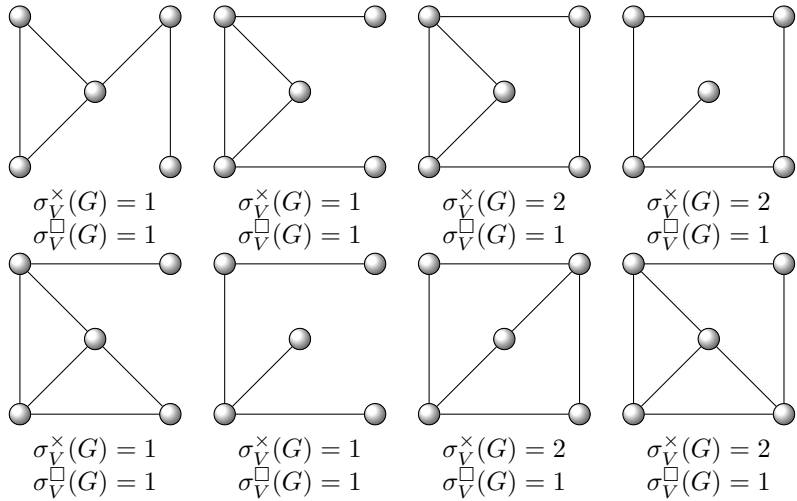
\captionof{figure}{ Connected graphs of order $5$ with radius $2$, other than $P_5$ and $C_5$} 
\label{pr3}
\end{center}
So, from these observations it is obvious that such graph needs to have at least $6$ vertices. And that proves our claim.
\end{proof}

Another class of graphs that we observed are binary trees. They are a useful class of graphs in observing hierarchical structures and are sometimes embedded in networks to help solve some network problems. They are also important in data science, because of the efficient way of searching through the data organized in this way. As searching through graph-like structures is somewhat similar to walking around that graph, we thought it might be useful to observe the spans of binary trees, for the possibility of $2$ simultaneous searches may arise. The structure for which we solved the span values is a perfect binary tree.
\textbf{A perfect binary tree} is a rooted binary tree in which all interior nodes have exactly $2$ children. The height of a binary tree is the number of non-root levels within the tree, and it is usually denoted by $h$. Height of the binary tree is also its radius. We will denote the perfect binary tree of height $h$ by $BT_h$.

\begin{theorem}\label{binary} 
It holds $$\sigma^{\boxtimes}_V(BT_h)=\sigma^{\times}_V(BT_h)=\sigma^{\square}_V(BT_h)=h-1.$$
\end{theorem}
\begin{proof}
First, we will prove that $\sigma^{\boxtimes}_V(BT_h)<h$, which will prove that all vertex spans are less than $h$, and then we will describe Alice's and Bob's walks through such graphs in a way that leaves them on a safety distance $h-1$ using different movement rules.\\
Let us denote the root of $BT_h$ by $x$. Alice and Bob can start their walks on the same side of $x$, on different sides of $x$ or one of them can start their walk in the vertex $x$. If they are on the same side, their distance on the starting point is already less than $h$. If they are on different sides, one of them will have to move to the other side, while the other one is still on that side or in $x$, as no movement rules allow them to simultaneously switch sides, because they have to pass through $x$ first. Which, brings us to the third option, one of them being on one side and the other in $x$. While one is in $x$, they can only be on distance $h$ if the other is in some leaf of the tree. But then, the only movement that will leave them in a safety distance $h$ will be to move from $x$ to the opposite side of the other actor. Which brings us back to the second situation. Therefore, it's impossible for Alice and Bob to walk through all the vertices of a perfect binary tree of height $h$ keeping safety distance $h$ at all times.\\
Now, let us describe Alice's and Bob's walks that keep them at safety distance $h-1$.\\
Let us name the two subgraphs induced by components of $BT_h-x$ as the left branch and the right branch.\\
Opposite lazy walks described:
Alice starts her walk in any leaf of the left branch and Bob in any leaf of the right branch. Alice now walks through all of the left branch while Bob stands still. Since he is in a leaf of the other branch Alice has never gotten closer than $h+1$ to him. Now she can go back to any of the leaves of the left branch so Bob can visit all vertices of the right branch. Now it is time to switch branches. While Alice is in a leaf of the left branch, Bob can go through the root $x$ and go down the left branch, always choosing a vertex that is more distant from Alice, till he reaches a leaf. That way he will be closest to her at the root of the left branch where his distance from Alice is exactly $h-1$. Now Alice can go up the left branch and switch to the right branch and go down to any leaf. Exactly as in the beginning, now both of them can in turns visit all the vertices in their respective branches, never getting closer to one another than $h-1$. So $\sigma^{\square}_V(BT_h)=h-1$. Also, since $h-1=\sigma^{\square}_V(BT_h) \leq \sigma^{\boxtimes}_V(BT_h) < h$, then $\sigma^{\boxtimes}_V(BT_h)=h-1$.\\
Active walks described:
Alice and Bob can visit all the vertices as they did in the opposite lazy walks, with few adjustments, as there is no more option for someone to stand still while the other one is walking. So when Alice is visiting all the vertices of the left branch, Bob can just switch from his leaf to the neighbour vertex and back. And the same goes for Alice when Bob is visiting his right branch vertices. And the other thing we need to keep in mind is that, when switching sides, when Bob is in the root of the left branch, Alice needs to be in a leaf, or they would come closer to each other than $h-1$, and for that to be possible, we just have to make sure to place Alice in the beginning of her walk in a leaf of the left branch if $h$ is odd and in a leaf neighbour vertex if $h$ is even. The illustration for $h=3$ is given in Figure \ref{binarno}.  

\begin{center}
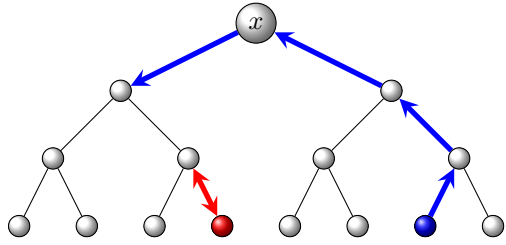

\begin{tikzpicture}[scale=0.9]
\node (1) at (0,0) [point] {};
\node (2) at (1,0) [point] {};
\node (3) at (2,0) [point] {};
\node (4) at (3,0) [pointred] {};
\node (5) at (4,0) [point] {};
\node (6) at (5,0) [point] {};
\node (7) at (6,0) [pointblue] {};
\node (8) at (7,0) [point] {};
\node (9) at (0.5,1) [point] {};
\node (10) at (2.5,1) [point] {};
\node (11) at (4.5,1) [point] {};
\node (12) at (6.5,1) [point] {};
\node (13) at (1.5,2) [point] {};
\node (14) at (5.5,2) [point] {};
\node (15) at (3.5,3) [point] {$x$};
\draw (1) -- (9) (2) -- (9) (3) -- (10)  (5) -- (11) (6) -- (11) (7) -- (12) (8) -- (12) (9) -- (13) (10) -- (13) (11) -- (14)  ;
\draw[-stealth, blue, line width=2pt] (7) -- (12);
\draw[-stealth, blue, line width=2pt] (12) -- (14);
\draw[-stealth, blue, line width=2pt] (15) -- (13);
\draw[-stealth, blue, line width=2pt] (14) -- (15);
\draw[stealth-stealth, red, line width=2pt] (4) -- (10);
\end{tikzpicture}
\captionof{figure}{Bob needs $4$ steps to get from the leaf of the right branch to the root of the left branch. At the same time Alice goes from the leaf to the neighbour and back, finishing on the leaf, keeping the distance $2$.} \label{binarno}
\end{center}
That way such switch of the positions is possible so the rest can take place as in lazy movement rules. So $\sigma^{\times}_V(BT_h)=h-1$.
\end{proof}


Based on the definition of an $n$-friendly graph and Observation 6.3. from \cite{banic} we observed different vertex span values for some classes of graphs. Our result from Theorem \ref{razlika} allowed us to expand these considerations. First observe that if there is a vertex $v$ in a graph $G$ of order $n$, such that $d(v)=n-1$,  then $\rad(G)=1$ and $\sigma^{\boxtimes}_{V}(G)\leq 1$. From this it is easily seen that for complete graphs $K_n$, $n\geq 3$, wheel graphs $W_n$, $n\geq 4$, and star graphs $S_n$, $n\geq 4$, the value of all spans will be equal to $1$. We present an overview of the other classes in Table \ref{familije}. 
\renewcommand{\arraystretch}{1.5}
\begin{center}
\captionof{table}{Span values for some classes of graphs}\label{familije}
\begin{tabular}{ |c|c|c|c|c|c| } 
 \hline
 Graph class & Radius & $\sigma^{\boxtimes}_{V}$ & $\sigma^{\times}_{V}$ & $\sigma^{\square}_{V}$ & Comment \\ 
 \hline
 $P_{n}$, $n\geq 2$ & $\lfloor\frac{n}{2}\rfloor$ & 1 & 1 & 0 & \cite{banic}\\ 
 $C_{n}$, $n\geq 3$ & $\lfloor\frac{n}{2}\rfloor$ &$\lfloor\frac{n}{2}\rfloor$ & $\lfloor\frac{n}{2}\rfloor$ &
$\left\{ \begin{array}{lc}
     \lfloor\frac{n}{2}\rfloor, & n \text{ odd} \\
     \frac{n}{2}-1, & n \text{ even}
\end{array} \right.$ & \cite{banic}, Thm \ref{razlika}\\ 
 $Q_{n}$, $n\geq 2$ & $n$ &$n$ & $n$ & $n-1$ & \cite{banic}, Thm \ref{razlika}\\ 
 $K_{r,s}$, $r,s\geq 2$ & $2$ & $2$ & $2$ & $1$ & \cite{banic}, Thm \ref{razlika}\\
 $PC_{n}$ & $\left \lfloor\frac{n}{2} \right \rfloor+1$ & $\left  \lceil\frac{n}{2} \right \rceil$ & $\left \lfloor\frac{n}{2} \right \rfloor$ & $\left  \lceil\frac{n}{2} \right \rceil$ & Thm \ref{paramecium}\\
 $BT_{h}$ & $h$ & $h-1$ & $h-1$ & $h-1$ & Thm \ref{binary}\\
 \hline
\end{tabular}
\end{center}

\section{Conclusion}
Inspired by the work in \cite{banic}, we observed the same problem in a somewhat different way. With that, we obtained a result for the relation between direct and Cartesian vertex spans and using this result we calculated spans for some graph classes. We saw that in the majority of the observed cases, direct vertex span is greater or equal to the Cartesian vertex span, but we also found a family for which the opposite holds. It would be interesting to find a characterization for the graphs for which Cartesian vertex span is greater than their direct vertex span and it would be interesting to discuss what is necessary in the graph structure for greater safety distance to be achieved by lazy movement rules instead of active movement rules, and vice versa. Further work may also include generalizing the notion of spans to more than two actors or analyzing the minimal number of steps in which the maximal safety distance can be achieved.

\end{document}